\documentclass[reqno]{amsart}
\usepackage[utf8]{inputenc}
\usepackage{amssymb,amsthm,amsfonts}
\usepackage{amsmath}
\usepackage{tikz}
\usepackage{graphicx}
\usepackage{psfrag}
\usepackage{cancel}
\usepackage{hyperref}
\newtheorem{definition}{Definition}[section]
\newtheorem{theorem}{Theorem}[section]
\newtheorem{lemma}{Lemma}[section]

\newtheorem{corollary}{Corollary}[section]
\newtheorem{remark}{Remark}[section]

\newtheorem{fact}{Fact}[section]
\newtheorem{claim}{Claim}[section]
\numberwithin{equation}{section}
\newcommand{\nt}{\ensuremath{\mathbb{N}}}
\newtheorem{innercustomthm}{Cross-Ratio Inequality\!}
\newenvironment{customthm}[1]
  {\renewcommand\theinnercustomthm{#1}\innercustomthm}
  {\endinnercustomthm}

\theoremstyle{plain}
\newtheorem{maintheorem}{Theorem}

\newcommand{\crd}{\operatorname{CrD}}
\newcommand{\dist}{\operatorname{Dist}}
\newcommand{\car}{\operatorname{Card}}
\newcommand{\crit}{\operatorname{Crit}}

\title[Beau bounds for multicritical circle maps]{Beau bounds for multicritical circle maps}

\author{Gabriela Estevez}
\address{Instituto de Matem\'atica e Estat\'istica, Universidade de S\~ao Paulo}
\curraddr{Rua do Mat\~ao 1010, 05508-090, S\~ao Paulo SP, Brasil}
\email{gestevez@ime.usp.br}

\author{Edson de Faria}
\address{Instituto de Matem\'atica e Estat\'istica, Universidade de S\~ao Paulo}
\curraddr{Rua do Mat\~ao 1010, 05508-090, S\~ao Paulo SP, Brasil}
\email{edson@ime.usp.br}

\author{Pablo Guarino}
\address{Instituto de Matem\'atica e Estat\'istica, Universidade Federal Fluminense}
\curraddr{Rua M\'ario Santos Braga S/N, 24020-140, Niter\'oi, Rio de Janeiro, Brasil}
\email{pablo\_\,guarino@id.uff.br}



\thanks{This work has been partially supported by ``Projeto Tem\'atico Din\^amica em Baixas Dimens\~oes'' FAPESP Grant 2011/16265-2, by
CAPES (PROEX) and by IMPA's summer scholar program}

\begin{document}

\begin{abstract} 
Let $f: S^1\to S^1$ be a $C^3$ homeomorphism without periodic points having a finite number of critical points of {\it power-law type\/}. 
In this paper we establish {\it real a-priori bounds\/}, on the geometry of orbits of $f$, which are \emph{beau} in the sense of Sullivan, {\it i.e.\/} bounds that are asymptotically universal at small scales. The proof of the beau bounds presented here is an adaptation, to the multicritical setting, of the one given by the second author and de Melo in \cite{dFdM}, for the case of a single critical point.
\end{abstract}

\maketitle

\begin{center}
\dedicatory{\it Dedicated to Sebastian van Strien on the occasion of his 60th birthday.}
\end{center}

\section{Introduction}\label{sec:intro} 

In the study of smooth dynamical systems, it is often the case that the geometry of orbits at fine scales is completely determined by a small number 
of dynamical invariants.  The invariants in question can be combinatorial, topological, even 
measure-theoretic. This phenomenon is known as {\it rigidity\/}. In general, since in many cases a smooth self-map has a plethora of periodic orbits 
whose eigenvalues can vary under small perturbations, and since such 
eigenvalues are smooth conjugacy invariants, one can only hope to have rigidity in the absence of periodic points. 

The greatest success in the study of rigidity of dynamical systems, so far, has been achieved in dimension one, {\it i.e.\/} for
interval or circle dynamics. This success has been most complete in the case of invertible smooth dynamics on the circle -- homeomorphisms or 
diffeomorphisms of $S^1$ with sufficient 
smoothness. Here, it is known since Poincar\'e and Denjoy that the only topological invariant is the rotation number.
It follows from the seminal works of M.~Herman \cite{hermanihes} and J.-C. Yoccoz \cite{Y} that if $f$ is a $C^r$-smooth diffeomorphism of $S^1$, 
with $r\geq 3$, whose rotation number $\rho$ satisfies the Diophantine condition 
\begin{equation}\label{defdiof}
\left| \rho -\frac{p}{q}\right|\;\geq\;  \frac{C}{q^{2+\beta}}
\end{equation}for all rational numbers $p/q$, for some constants $C>0$ and $0\leq\beta<1$, then $f$ is $C^{r-1-\beta-\epsilon}$-conjugate to 
the corresponding rigid rotation, for every $\epsilon>0$. In other words, for almost all rotation numbers, a sufficiently smooth circle 
diffeomorphism is almost as smoothly conjugate to the rotation with the same rotation number. The slight loss of differentiability for the 
conjugacy is inherent to small-denominator problems, and is already present even if the diffeomorphism in question is a small perturbation 
of a rotation. On the other hand, Arnol'd has shown in his thesis (see \cite{Ar}) that there are real-analytic circle diffeomorphisms with ``bad" irrational
rotation numbers which are not even absolutely continuously conjugate to the corresponding rotation. These results (enhanced by further 
developments, {\it e.g.,\/} \cite{KO} and \cite{KT2}) yield a fairly complete solution to the rigidity problem for circle diffeomorphisms. 

For smooth homeomorphisms of the circle with critical points (of non-flat type), the topological classification is due to Yoccoz \cite{yoccoz}, see Theorem \ref{yoccoztheorem} below. Since no conjugacy between a map of this kind and the corresponding rigid rotation can be smooth, 
the correct thing to do when studying rigidity is to compare two such maps directly. In other words, assuming that there exists a topological 
conjugacy between two such maps taking the critical points of one to the critical points of the other, one asks: is this conjugacy a smooth diffeomorphism? 

In the case of smooth homeomorphisms having exactly one critical point -- the so-called {\it critical circle maps\/} -- a reasonably complete rigidity theory has emerged in recent years, thanks to the combined efforts of several mathematicians -- see \cite{dFdM, dFdM2, KT, KY} for the case of real-analytic homeomorphisms, and \cite{Gua,GMdM,GdM} for the case of \emph{finitely} smooth homeomorphisms. We summarize those contributions in the following statements: any two $C^3$ circle homeomorphisms with the same irrational rotation number of \emph{bounded type} (that is, $\beta=0$ in \eqref{defdiof}) and with a unique critical point (of the same odd power-law type), are conjugate to each other by a $C^{1+\alpha}$ circle diffeomorphism, for some universal $\alpha>0$ \cite{GdM}. Moreover, any two $C^4$ circle homeomorphisms with the same irrational rotation number and with a unique critical point (again, of the same odd type), are conjugate to each other by a $C^1$ diffeomorphism \cite{GMdM}. As it turns out, this 
conjugacy is a $C^{1+\alpha}$ diffeomorphism for a certain set of rotation numbers that has full Lebesgue measure (see \cite[Section 4.4]{dFdM} for its 
definition), but {\it does not\/} include all irrational rotation numbers (see the counterexamples in \cite{Av} and \cite[Section 5]{dFdM}).

By contrast, for smooth circle homeomorphisms having two or more critical points -- the so-called {\it multicritical circle maps\/}, see Definition \ref{def:multicritic} below -- the rigidity problem remains wide open.

The very first step in the study of rigidity is to get {\it real a-priori bounds\/} on the geometry of orbits. As it turns out, 
for one-dimensional maps with a finite number of critical points, the behaviour of the critical orbits essentially determines the behaviour
of all other orbits. Hence the task reduces to finding a-priori bounds on the critical orbits, and for this it suffices to get uniform 
bounds on the sequence of {\it scaling ratios\/} around each critical point, determined by succesive {\it closest returns\/} of the forward 
orbit of the critical point to itself. See \S \ref{prelim} for the relevant definitions. Such bounds have been obtained by M.~Herman
\cite{H} and G.~\'Swi\c{a}tek \cite{S}. A detailed proof of such bounds in the case of multicritical circle maps can be found in \cite{EdF}. 

Our goal in the present paper is to improve on the bounds presented in \cite{EdF} by showing that they are {\it beau\/} in the sense of 
Sullivan \cite{Su} (see Theorem \ref{teobeau} in Section \ref{statements}). This means that such bounds on the scaling ratios of the critical
orbits become asymptotically universal, {\it i.e.\/} 
independent of the map. As in the case of maps with a single critical point, beau bounds should yield a strong form of compactness of the 
{\it renormalizations\/} of a given multicritical circle map. However, the precise definition of such 
{`renormalization semi-group'\/} in the multicritical case is combinatorially more elaborate, and  
is therefore beyond the scope of the present paper.


\subsection{Summary of results} We proceed to {\it informal statements\/} of our main results. Some rough explanations about the terminology
adopted in these statements are in order (precise statements will be given in Section \ref{statements}). We write $S^1=\mathbb{R}/\mathbb{Z}$
for the unit circle, taken as an affine $1$-manifold, and use additive notation throughout. By a {\it multicritical circle map\/} 
$f: S^1\to S^1$ we mean an orientation-preserving, $C^3$-smooth homeomorphism having a finite number of critical points, all of which are 
non-flat (of power-law type), see Definition \ref{def:multicritic} below. Only maps without periodic orbits will matter to us. By a 
{\it scaling ratio\/} around a critical point we mean the ratio of distances, to said critical point, of two consecutive 
{\it closest returns\/} of the forward orbit of that critical point. 

\begin{theorem}\label{scalingratios} Let $f: S^1\to S^1$ be a multicritical circle map with irrational rotation number. Then the successive 
scaling ratios around each critical point of $f$ are uniformly bounded, and the bound is asymptotically independent of $f$.
\end{theorem}


This theorem is, in fact, a special case of Theorem \ref{teobeau} stated in Section \ref{statements}; see also Section \ref{sec:beaubounds}. 
The main consequence of this result is the following {\it quasi-symmetric rigidity\/} statement, which is an improvement over the main 
theorem in \cite{EdF}. Given an orientation-preserving homeomorphism $h: S^1\to S^1$, we define its {\it local quasi-symmetric distortion\/}
to be the function $\sigma_h: S^1\to \mathbb{R}^+\cup \{\infty\}$ given by
\[
 \sigma_h(x)\;=\; \lim_{\delta\to 0}\limsup_{|t|\leq \delta} \frac{|h(x+t)-h(x)|}{|h(x)-h(x-t)|}\ .
\]
When $\sigma_h(x)\leq M$ for all $x \in S^1$ and some constant $M\geq 1$, we say that $h$ is {\it quasi-symmetric\/}. 

\begin{corollary}\label{coro} Let $f,g:S^1\to S^1$ be multicritical circle maps with the same irrational rotation number and the same number
$N$ of (non-flat) critical points, whose criticalities are bounded by some number $d>1$. Suppose $h: S^1\to S^1$ is a topological conjugacy 
between $f$ and $g$ which maps each critical point of $f$ to a critical point of $g$. Then $h$ is quasi-symmetric, and its local 
quasi-symmetric distortion is universally bounded, {\it i.e.\/} there exists a constant $K=K(N,d)>1$, independent of $f$ and $g$, such that 
$\sigma_h(x)\leq K$ for all $x\in S^1$.
\end{corollary}

The fact that the conjugacy $h$ is {\it quasi-symmetric\/}, in the above corollary, is the main theorem proved in \cite{EdF}. The improvement 
here is that the quasi-symmetric distortion of $h$ is {\it asymptotically universal\/}. We provide a sketch of the proof of 
Corollary \ref{coro} in Section \ref{sec:beaubounds}.

This paper is organized as follows: in Section \ref{prelim}, we recall some well-know facts and state our main
results (see Section \ref{statements}). In Section \ref{secC1bounds} we establish $C^1$-bounds for suitable return maps around a critical 
point, while in Section \ref{secnegscwarz} we prove that these return maps have negative Schwarzian derivative. In Section 
\ref{sec:beaubounds} we prove Theorem \ref{teobeau}, Theorem \ref{CRIuniversal}, Theorem \ref{scalingratios} and Corollary \ref{coro}. 
Finally, in Appendix \ref{appA}, we provide proofs of some auxiliary results stated and used along the text.

\section{Preliminaries and statements of results}\label{prelim}

In this section we review some classical tools of one-dimensional dynamics that will be used along the text, and we state our main results
(see Section \ref{statements}).

\subsection{Cross-ratios} Given two intervals $M\subset T\subset S^{1}$ with $M$ compactly contained in $T$ (written $M\Subset T$) let us denote by $L$ and $R$ the two connected components of $T\setminus M$. We define the \emph{cross-ratio} of the pair $M,T$ as follows:
\begin{equation*}
[M,T]= \frac{|L|\,|R|}{|L\cup M|\,|M \cup R|} \in (0,1).
\end{equation*}

The cross-ratio is preserved by M\"obius transformations. Moreover, it is weakly contracted by maps with negative Schwarzian derivative (see Lemma \ref{contracts} below).

Let $f:S^{1}\to S^{1}$ be a continuous map, and let $U\subseteq S^{1}$ be an open set such that $f|_{U}$ is a 
homeomorphism onto its image. If $M\subset T\subset U$ are intervals, with $M\Subset T$,  
the \textit{cross-ratio distortion} of the map $f$ on the pair of intervals $(M,T)$ is defined to be the ratio
\begin{equation*}
\crd(f;M,T)= \frac{\big[f(M),f(T)\big]}{[M,T]}.
\end{equation*}

If $f|T$ is a M\"obius transformation, then we have that $\crd(f;M,T)=1$. When $f|T$ is a diffeomorphism onto its image and $\log{Df}|T$ has {\it bounded variation\/} in $T$ (for instance, if $f$ is a $C^2$ diffeomorphism), we obtain $\crd(f;M,T)\leq e^{2V}$, where $V=\mathrm{Var}(\log{Df}|T)$.   
We shall use the following chain rule in iterated form:
\begin{equation}\label{CRIchainrule}
 \crd(f^j;M,T) = \prod_{i=0}^{j-1} \crd(f;f^{i}(M), f^{i}(T))\ .
\end{equation}

\subsection{Distortion and the Schwarzian}
If $f:T \to f(T)$ is a $C^1$ diffeomorphism, we define its \textit{distortion} by 
\[\dist(f,T)=\sup_{x,y \in T}\dfrac{|Df(x)|}{|Df(y)|}\,.\]

Note that $\dist(f,T)=1$ if, and only if, $f$ is an affine map on $T$. In any other case we have $\dist(f,T)>1$. By the Mean Value Theorem, 
we have the following fact. 

\begin{remark}\label{remdistandcrd} If $\dist(f,T)<1+ \varepsilon$, then $\crd(f;M,T)<(1+\varepsilon)^2$ for any $M \subset T$.
\end{remark}

Recall that for a given $C^3$ map $f$, the \textit{Schwarzian derivative} of $f$ is the differential operator defined for all $x$ 
 regular point of $f$ by:
 \begin{equation*}
  Sf(x)= \dfrac{D^{3}f(x)}{Df(x)} - \dfrac{3}{2} \left( \dfrac{D^{2}f(x)}{Df(x)}\right)^{2}.
 \end{equation*}

The relation between the Schwarzian derivative and cross-ratio distortion is given by the following well known fact.

\begin{lemma}\label{contracts} If $f$ is a $C^3$ diffeomorphism with $Sf<0$, then for any two intervals $M\subset T$ contained in the domain
of $f$ we have $\crd(f;M,T)<1$, that is, $\big[f(M),f(T)\big]<[M,T]$.
\end{lemma}

See Appendix \ref{appA} for a proof.

\subsection{Multicritical circle maps} 
Let us now define the maps which are the main object of study in the present paper. We start with the notion of {\it non-flat critical point\/}. 

\begin{definition}\label{naoflat} We say that a critical point $c$ of a one-dimensional $C^3$ map $f$ is \emph{non-flat} of degree $d>1$ if there exists 
a neighborhood $W$ 
of the critical point such that $f(x)=f(c)+\phi(x)\big|\phi(x)\big|^{d-1}$ for all $x \in W$, where $\phi : W \rightarrow \phi(W)$ is 
a $C^{3}$ diffeomorphism such that $\phi(c)=0$. The number $d$ is also called the \textit{criticality}, 
the \textit{type} or the \textit{order} of $c$. 
\end{definition}

We recall here, the following facts about the geometric behaviour of a map near a non-flat critical point. 

\begin{lemma} \label{lemacalculo} Given $f$ with a non-flat critical point $c$ of degree $d>1$ there exists a neighborhood 
$U \subseteq W$ of $c$ such that:
\begin{enumerate}
 \item\label{itemSf} $f$ has negative Schwarzian derivative on $U \setminus \{ c \}$. More precisely, there exists $K=K(f)>0$ such that
 for all $x \in U \setminus \{ c \}$ we have: $$Sf(x)< -\frac{K}{(x-c)^2}\,.$$ 
 \item\label{itemDf} There exist constants $0<\alpha< \beta$ such that for all $x \in U$
 \begin{equation*}
  \alpha|x-c|^{d-1} < Df(x) < \beta |x-c|^{d-1}.
 \end{equation*}
 Moreover, $\alpha$ and $\beta$ can be chosen so that $\beta<(3/2)\alpha$.
 \item\label{itemdist} Given a non-empty interval $J \subseteq U$ and $x \in J$ we have
 \begin{equation*}
  Df(x) \leq 3d\,\dfrac{|f(J)|}{|J|}\,.
 \end{equation*}
 \item\label{itemcross} Given two non-empty intervals $M \subseteq T \subseteq U$ we have:$$\crd(f;M,T)\leq 9d^2\,.$$  
\end{enumerate}
\end{lemma}

We postpone the proof of Lemma \ref{lemacalculo} to Appendix \ref{appA}.

\begin{definition}\label{def:multicritic} A \emph{multicritical circle map} is an orientation preserving $C^3$ circle homeomorphism having 
$N \geq 1$ critical points, all of which are non-flat in the sense of Definition \ref{naoflat}.
\end{definition}

Being a homeomorphism, a multicritical circle map $f$ has a well defined rotation number. We will focus on the case when $f$ has no periodic orbits. 
By a result of J.-C. Yoccoz \cite{yoccoz}, $f$ has no wandering intervals. More precisely, we have the following fundamental result. 

\begin{theorem}[Yoccoz \cite{yoccoz}]\label{yoccoztheorem} Let $f$ be a multicritical circle map with irrational rotation number $\rho$. 
Then $f$ is topologically conjugate to the rigid rotation $R_{\rho}$, i.e., there exists a homeomorphism $h: S^{1} \rightarrow S^{1}$ such that $h \circ f = R_{\rho} \circ h.$
\end{theorem}

Given a family of intervals $\mathcal{F}$ on $S^{1}$ and a positive integer $m$, we say that $\mathcal{F}$ has {\it multiplicity of intersection 
at most $m$\/} if each $x\in S^{1}$ belongs to at most $m$ elements of $\mathcal{F}$. 

\begin{customthm}{} Given a multicritical critical circle map $f:S^1\to S^1$, there exists a constant $C>1$, depending only on $f$, 
such that the following holds. If $M_i\Subset T_{i} \subset S^1$, where $i$ runs through some finite set of indices $\mathcal{I}$, 
are intervals on the circle such that the family $\{T_i: i\in \mathcal{I}\}$ 
has multiplicity of intersection at most $m$, then 
 \begin{equation}\label{crossprod}
  \prod_{i \in \mathcal{I}} \crd(f;M_{i},T_{i}) \leq C^{m}\ .
 \end{equation}
\end{customthm}

The Cross-Ratio Inequality was obtained by \'Swi\c{a}tek in \cite{S}. Similar estimates were obtained before by Yoccoz in \cite{yoccoz}, 
on his way to proving Theorem \ref{yoccoztheorem} (see \cite[Chapter IV]{dMvS} for this and much more). In this paper we will improve the Cross-Ratio Inequality, obtaining \emph{universal} bounds (see Theorem \ref{CRIuniversal} in Section \ref{statements}).

\bigskip

As explained before, given two intervals $M\subset T\subset S^{1}$ with $M\Subset T$ (that is, $M$ is compactly contained in $T$), we denote by $L$ and $R$ the two connected components of $T\setminus M$. We define the \emph{space} of $M$ inside $T$ as the smallest of the ratios $|L|/|M|$ and $|R|/|M|$. If the space is $\tau>0$ we said that $T$ contains a \emph{$\tau$-scaled neighborhood} of $M$.

\begin{lemma}[Koebe distortion principle] \label{koebe}
For each $\ell,\tau>0$ and each multicritical circle map $f$ there exists a 
constant $K=K(\ell,\tau,f)>1$ of the form
\begin{equation}\label{constkoebe}
 K= \left(1 + \frac{1}{\tau} \right)^2 \exp (C_0\,\ell)\,,
\end{equation}
where $C_0$ is a constant depending only on $f$, with the following property. If
$T$ is an interval such that $f^k|_{T}$ is a diffeomorphism onto its image and if $\sum_{j=0}^{k-1} |f^j(T)|\leq \ell$, 
then for each interval $M\subset T$ for which $f^k(T)$ contains a $\tau$-scaled neighborhood
of $f^k(M)$ one has
\[
\frac{1}{K}\leq \frac{|Df^k(x)|}{|Df^k(y)|}\leq K
\]
for all $x,y\in M$. 
\end{lemma}

A proof of the Koebe distortion principle can be found in \cite[p.~295]{dMvS}. 

\subsection{Combinatorics and real bounds}
Let $f$ be  a multicritical circle map, and let $c_0,c_1,\ldots, c_{N-1}$ be its critical points. 
As already mentioned in the introduction, we assume throughout that $f$ has no periodic points. Let $\rho$ be the rotation number of $f$. As we know, it has a infinite continued fraction expansion, say
\begin{equation*}
      \rho(f)= [a_{0} , a_{1} , \cdots ]=   
      \cfrac{1}{a_{0}+\cfrac{1}{a_{1}+\cfrac{1}{ \ddots} }} \ .
    \end{equation*}
    
A classical reference for continued-fraction expansion is the monograph \cite{khin}. Truncating the expansion at level $n-1$, we obtain a sequence
of fractions $p_n/q_n$ which are called the \emph{convergents} of the irrational $\rho$.
$$
\frac{p_n}{q_n}\;=\;[a_0,a_1, \cdots ,a_{n-1}]\;=\;\dfrac{1}{a_0+\dfrac{1}{a_1+\dfrac{1}{\ddots\dfrac{1}{a_{n-1}}}}}\ .
$$
Since each $p_n/q_n$ is the best possible approximation to $\rho$ by fractions with denominator at most $q_n$ \cite[Chapter II, Theorem 15]{khin},
we have:
\begin{equation*}
\text{If} \hspace{0.4cm} 0<q<q_n \hspace{0.3cm} \text{then} \hspace{0.3cm} \left|\rho-\frac{p_n}{q_n}\right|<\left|\rho-\frac{p}{q}\right|, 
\hspace{0.3cm} \text{ for any $p\in\nt$.}
\end{equation*} 
The sequence of numerators satisfies
\begin{equation*}
 p_0=0,  \hspace{0.4cm} p_{1}=1, \hspace{0.4cm} p_{n+1}=a_{n}p_{n}+p_{n-1} \hspace{0.3cm} \text{for $n \geq 1$}.
\end{equation*}
Analogously, the sequence of the denominators, which we call the \emph{return times}, satisfies 
\begin{equation*}
 q_{0}=1, \hspace{0.4cm} q_{1}=a_{0}, \hspace{0.4cm} q_{n+1}=a_{n}q_{n}+q_{n-1} \hspace{0.3cm} \text{for $n \geq 1$} .
\end{equation*}
For each point $x \in S^{1}$, the closed interval with endpoints $x$ and $f^{q_{n}}(x)$ containing the point $f^{q_{n+2}}(x)$ 
contains no other iterate $f^{j}(x)$ with $1 \leq j \leq q_{n}-1$. 

For each critical point $x \in S^1$ and each non-negative integer $n$, let $I_{n}(x)$ be the interval with endpoints 
$x$ and $f^{q_n}(x)$ containing  $f^{q_{n+2}}(x)$.. We write $I_{n}^{j}(x)=f^{j}(I_{n}(x))$ for all $j$ and $n$. 

\begin{lemma}\label{lemapartition} For each $n\geq 0$ and each $x\in S^1$, the collection of intervals
\[
 \mathcal{P}_n(x)\;=\; \left\{f^i(I_n(x)):\;0\leq i\leq q_{n+1}-1\right\} \;\bigcup\; \left\{f^j(I_{n+1}(x)):\;0\leq j\leq q_{n}-1\right\} 
\]
is a partition of the unit circle (modulo endpoints), called the {\it $n$-th dynamical partition\/} associated to the point $x$.
\end{lemma}

See Appendix \ref{appA} for a proof. \\

For each $n$, the partition $\mathcal{P}_{n+1}(x)$ is a (non-strict) refinement of $\mathcal{P}_{n}(x)$, while the partition 
$\mathcal{P}_{n+2}(x)$ is a strict refinement of $\mathcal{P}_{n}(x)$.  

Let us focus our attention on one of the critical points only, say $c_0$, and on its associated dynamical partitions, namely $\mathcal{P}_{n}(c_{0})$ 
($n\geq 0$). 
To simplify the notation, we shall write below $\mathcal{P}_{n}$ instead of $\mathcal{P}_{n}(c_{0})$. Accordingly, the intervals $I_{n}^i(c_0)$
and $I_{n+1}^j(c_0)$ will be denoted by $I_{n}^i$ and $I_{n+1}^j$, respectively. Moreover, for a given $J\in\mathcal{P}_n$ 
we shall denote by $J^*$ the union of $J$ with its left and right neighbours in $\mathcal{P}_n$.
We may assume, for $n$ large enough, that no two critical points of $f$ are in the same atom of $\mathcal{P}_{n}$.

 \begin{theorem}[Real A-priori Bounds] \label{realbounds} 
 Let $f$ be a multicritical circle map. There exists a constant $C>1$ depending only of $f$ such that the following holds. 
For all $n\geq 0$ and for each pair of adjacent atoms $I, J\in \mathcal{P}_{n}$ we have
 \begin{equation}\label{compatom}
  C^{-1} {|J|} \leq |I| \leq C|J|.
 \end{equation}
In particular there exists $\mu=\mu(f) \in (0,1)$ such that, if $\mathcal{P}_{n+2}\ni \Delta \subset \Delta'\in \mathcal{P}_{n}$, then 
$|\Delta|<\mu |\Delta'|$ for all $n \in \nt$.
 \end{theorem}

When we get the inequalities in \eqref{compatom} for two atoms $I$ and $J$, we will say that they are {\it comparable\/}, which will be denoted by $|I| \asymp |J|$. Thus the above theorem is saying that {\it any two adjacent atoms of a dynamical partition of $f$ are comparable\/}.\\

Note that for a rigid rotation we have $|I_n|=a_{n+1}|I_{n+1}|+|I_{n+2}|$. If $a_{n+1}$ is big, then $I_n$ is much larger than $I_{n+1}$. Thus, even for rigid rotations, real bounds do not hold in general.

Theorem \ref{realbounds} was obtained by Herman \cite{H}, based on estimates by \'Swi\c{a}tek \cite{S}. Further proofs are to be
found in \cite{dFdM} and \cite{P} for the case of a single critical point, and in \cite{EdF} for the general case.

\subsection{Statements of main results}\label{statements} Our main goal in the present paper is to establish the following two results,
which immediately imply Theorem \ref{scalingratios}.

\begin{maintheorem}[Beau bounds]\label{teobeau} Given $N\geq 1$ in $\nt$ and $d>1$ there exists a constant $B=B(N,d)>1$ with
the following property: given a multicritical circle map $f$, with at most $N$ critical points whose criticalities are bounded by $d$, there
exists $n_0=n_0(f)\in\nt$ such that for all $n \geq n_0$ and for any adjacent intervals $I$ and $J$ in $\mathcal{P}_{n}$
 we have:
 \[ 
  \dfrac{|J|}{B} \leq |I| \leq B|J|.
 \]
\end{maintheorem}

\begin{maintheorem}\label{CRIuniversal} Given $N\geq 1$ in $\nt$ and $d>1$ there exists a constant $B=B(N,d)>1$ with
the following property: given a multicritical circle map $f$, with at most $N$ critical points whose criticalities are bounded by $d$, there
exists $n_0=n_0(f)$ such that for all $n \geq n_0$, $\Delta\in\mathcal{P}_n$ and $k\in\nt$ such that $f^j(\Delta)$ is contained in an
element of $\mathcal{P}_n$ for all $1\leq j \leq k$, we have that:$$\crd(f^k;\Delta,\Delta^*)\leq B\,.$$
\end{maintheorem}

The proof of the beau bounds (Theorem \ref{teobeau}) is the same as the proof of the real bounds (Theorem \ref{realbounds}) given by the 
first two authors in \cite[Section 3, p. 8-16]{EdF}, but replacing the Cross-Ratio Inequality with Theorem \ref{CRIuniversal}. In other
words, Theorem \ref{teobeau} follows directly from Theorem \ref{CRIuniversal}. The remainder of this paper is devoted to proving Theorem 
\ref{CRIuniversal}. Its proof will be given in Section \ref{sec:beaubounds}.
 
\section{The $C^1$ bounds}\label{secC1bounds}

In this section we prove the following result. 

\begin{lemma}\label{keyLemma} Given a multicritical circle map $f$ there exist two constants $K=K(f)>1$ and 
$n_{0}=n_{0}(f) \in \nt$ such that for all $n >n_{0}$, $x \in I_{n}$ and $j \in \{0,1, \cdots, q_{n+1} \}$, we have
 \begin{equation}\label{keylemma}
  Df^{j}(x) \leq K\,\dfrac{|f^{j}(I_{n})|}{|I_{n}|}\,.
 \end{equation}
\end{lemma}

For future reference, we note the following consequence of the real bounds.

\begin{corollary}\label{coroC1bounds} The sequence $\big\{f^{q_{n+1}}|I_{n}\big\}$ is bounded in the $C^1$ metric.
\end{corollary}

\begin{proof}[Proof of Corollary \ref{coroC1bounds}] By combinatorics, $I_{n+1} \subset f^{q_{n+1}}(I_n) \subset I_n \cup I_{n+1}$. 
Then:$$\frac{|I_{n+1}|}{|I_{n}|}\leq\frac{\big|f^{q_{n+1}}(I_n)\big|}{|I_{n}|}\leq 1+\frac{|I_{n+1}|}{|I_{n}|}\,.$$
By the real bounds (Theorem \ref{realbounds}) we have $|I_{n+1}|\asymp|I_n|$, and then $\big|f^{q_{n+1}}(I_{n})\big|\asymp|I_{n}|$. Therefore 
Corollary \ref{coroC1bounds} follows from Lemma \ref{keyLemma}.
\end{proof}

The remainder of this section is devoted to proving Lemma \ref{keyLemma}.

\begin{proof}[Proof of Lemma \ref{keyLemma}] For each $n \in \nt$ consider $L_{n}=I_{n+1}$, $R_{n}= f^{q_{n}}(I_n)$ and $T_{n}=I_{n}^{*}= L_{n} \cup I_{n} \cup R_{n}$. We have three preliminary facts:

\begin{fact}\label{FATO1} The family $\{ T_{n}, f(T_{n}), \cdots, f^{q_{n+1}-1}(T_{n}) \}$ has intersection multiplicity bounded by $3$.
\end{fact}

Fact \ref{FATO1} follows from the following general fact: given $z\in S^1$ and $n\in\nt$ let $I=\big[z,R_{\rho}^{3q_n}(z)\big]$, where $R_{\rho}$
is the rigid rotation of angle $2\pi\rho$ in the unit circle. Then the multiplicity of intersection of the family 
$\big\{I,R_{\rho}(I),...,R_{\rho}^{q_{n+1}-1}(I)\big\}$ is $3$ for any $n\in\nt$.

\begin{fact}\label{FATO2} There exists a constant $\tau>0$ (depending only on the real bounds of $f$) such
 that$$|L_n^j|>\tau|I_n^j|\quad\mbox{and}\quad|R_n^j|>\tau|I_n^j|$$for each $j \in \{ 0, \cdots, q_{n+1} \}$ and
 for all $n\in\nt$.
 \end{fact}

\begin{proof}[Proof of Fact \ref{FATO2}] For $j=0$, observe that the intervals $L_n$, $I_n$ and $R_n$ are adjacent and belong to the dynamical 
partition $\mathcal{P}_n$, then by the real bounds they are comparable by a constant that only depends on $f$. Let us prove now that for 
$j= q_{n+1}$ the three intervals $L_n^{j}$, $I_n^j$ and $R_n^j$ are comparable too.

On one hand, the intervals $I_{n+1}$ and $I_{n+1}^{q_{n+1}}$ are adjacent and belong to $\mathcal{P}_{n+1}$, then they are comparable 
(again by the real bounds).
Moreover $I_{n+1}\subset I_{n}^{q_{n+1}} \subset I_{n+1}\cup I_n$. By the real bounds $|I_n|\asymp|I_{n+1}|$ and then
$|I_{n}^{q_{n+1}}|\asymp|I_{n+1}^{q_{n+1}}|$, that is:
\begin{equation} \label{fact2}
|L_{n}^{q_{n+1}}| \asymp |I_{n}^{q_{n+1}}|\,.
\end{equation}

On the other hand, the intervals $I_{n}$ and $I_{n}^{q_{n}}$ are adjacent and belong to $\mathcal{P}_{n}$, then they are comparable.
Moreover:$$I_{n+1}^{q_{n}}\subset I_{n}^{q_n+q_{n+1}}\subset I_n \cup I_{n}^{q_{n}}\,.$$

From \cite[item (v), p. 14]{EdF} we know that $|I_{n+1}^{q_{n}}| \asymp |I_{n}|$ and then $|I_{n}^{q_{n}+q_{n+1}}| \asymp |I_{n}|$. But
$I_{n+1}\subset I_n^{q_{n+1}}\subset I_n\cup I_{n+1}$ and then by the real bounds:
\begin{equation} \label{fact22}
 |R_{n}^{q_{n+1}}|=|I_{n}^{q_{n}+q_{n+1}}| \asymp |I_{n}| \asymp |I_n^{q_{n+1}}|\,.
\end{equation}

Therefore, for $j= q_{n+1}$, the three intervals $L_n^{j}$, $I_n^j$ and $R_n^j$ are comparable. Now, let $1 \leq j \leq q_{n+1}-1$. Consider the intervals $|L_{n}^j|, \ |I_{n}^j| , \ |R_{n}^j|$ and their images by the map $f^{q_{n+1}-j}$. By the Cross-Ratio Inequality (combined with Fact \ref{FATO1}) we have that there exists a constant $K_0=K_0(f)>1$ such that
\begin{equation*}
  \dfrac{|L_{n}^{q_{n+1}}| |R_{n}^{q_{n+1}}| |L_{n}^{j} \cup I_{n}^{j}| |I_{n}^{j} \cup R_{n}^{j}|}{|L_{n}^{j}| |R_{n}^{j}| 
  |L_{n}^{q_{n+1}} \cup I_{n}^{q_{n+1}}| |I_{n}^{q_{n+1}} \cup R_{n}^{q_{n+1}}|} \leq K_0\,.
\end{equation*}

Using \eqref{fact2} and \eqref{fact22} in the last inequality, we get
\begin{equation*}
 \left( 1 + \dfrac{|I_{n}^j|}{|L_{n}^{j}|} \right) \left( 1 + \dfrac{|I_{n}^j|}{|R_{n}^j|} \right) \leq K\,,
\end{equation*}
and we are done.
\end{proof}

\begin{remark} We can always assume, whenever necessary, that $n_{0}=n_{0}(f)$ given by Lemma \ref{keyLemma} is such that for all $n\geq n_{0}$ and $j\in\{0,...,q_{n+1}\}$ we have $\car( f^{j}(T_{n}) \cap\crit(f)) \leq 1$, where $\car$ denotes the cardinality of a finite set, and $\crit(f)$ is the set of critical points of $f$ (this is because, by minimality, $\big|f^j(T_n)\big|$ goes to zero as $n$ goes to infinity).
\end{remark}

\begin{definition}[Critical times] We say that $j \in \{1, \cdots, q_{n+1} \}$ is a \emph{critical time} if 
$f^{j}(T_{n}) \cap\crit(f) \neq \O$.
\end{definition}

\begin{remark} Note that $\car( \{\mbox{critical times} \} )\leq 3N$.
\end{remark}

\begin{fact}\label{regbranches} Let $1 \leq j_1<j_2\leq q_{n+1}$ be two consecutive critical times. Then for all $x \in f^{j_1+1}(I_n)$ we have:
 $$Df^{j_2-j_1-1}(x) \asymp \frac{|f^{j_2}(I_n)|}{|f^{j_1+1}(I_n)|}\,,$$with universal constants (depending only on the real bounds).
\end{fact}

\begin{proof}[Proof of Fact \ref{regbranches}] 
Note that $f^{j_2-j_1-1}:f^{j_1+1}(T_n) \to f^{j_2}(T_n)$ is a diffeomorphism. Fact \ref{FATO1} implies that  
$\sum_{i=0}^{j_2-j_1-1}|f^{i}(f^{j_1+1}(T_n))| < 3$, and by Fact \ref{FATO2} the interval $f^{j_2-j_1-1}(f^{j_1+1}(T_n))$ contains a 
$\tau-$scaled neighborhood of $f^{j_2-j_1-1}(f^{j_1+1}(I_n))$. By Koebe Distortion Principle (Lemma \ref{koebe}) there exists a constant
$K_0=K_0(f)>1$ such that for all $x,y \in f^{j_1+1}(I_n)$ we have that
\begin{equation*}
 \dfrac{1}{K_0} \leq \dfrac{Df^{j_2-j_1-1}(x)}{Df^{j_2-j_1-1}(y)} \leq K_0\,.
\end{equation*}
Let $y\in I_n^{j_1+1}$ be given by the Mean Value Theorem such that 
\[
Df^{j_2-j_1-1}(y)=\dfrac{|f^{j_2}(I_n)|}{|f^{j_1+1}(I_n)|}\ .
\] 
Then for all $x \in f^{j_1+1}(I_n)$,
\begin{equation*}
 \dfrac{1}{K_0} \dfrac{|f^{j_2}(I_n)|}{|f^{j_1+1}(I_n)|} \leq Df^{j_2-j_1-1}(x) \leq K_0 \dfrac{|f^{j_2}(I_{n})|}{|f^{j_1+1}(I_n)|}\,. 
\end{equation*}
\end{proof}

We finish the proof of Lemma \ref{keyLemma} by combining Fact \ref{regbranches} and Item \eqref{itemdist} in Lemma \ref{lemacalculo} with the help of the 
chain rule: $$Df^{j}(x) \leq (3d)^{3N}K_0^{3N}\,\dfrac{|f^{j}(I_{n})|}{|I_{n}|}\quad\mbox{for any $x \in I_{n}$ and 
$j \in \{1, \cdots, q_{n+1} \}$\,,}$$where $N=\car\big(\crit(f)\big)$ is the number of critical points of $f$, $d$ is the maximum of its criticalities and $K_0=K_0(f)$ is given by Fact \ref{regbranches}.
\end{proof}

\section{The negative Schwarzian property}\label{secnegscwarz}

In this section we prove the following result.

\begin{lemma}\label{negschwarz} Let $f$ be a multicritical circle map. There exists $n_1=n_1(f)\in\nt$ such that for all $n \geq n_1$ we have
that
\[ Sf^{j}(x)<0\quad\text{for all $j\in \{1, \cdots, q_{n+1}\}$ and for all $x \in I_{n}$ regular point of $f^{j}$.}
 \]
Likewise, we have 
\[ Sf^{j}(x)<0\quad\text{for all $j\in \{1, \cdots, q_{n}\}$ and for all $x \in I_{n+1}$ regular point of $f^j$}.
 \]
\end{lemma}

In the proof we adapt the exposition in \cite[pages 380-381]{dFdM}.

\begin{proof}[Proof of Lemma \ref{negschwarz}] We give the proof only for the case $x\in I_n$ regular point of $f^j$ for some 
$j \in \{ 1, \cdots, q_{n+1}\}$ (the other case is entirely analogous).

By Item \eqref{itemSf} in Lemma \ref{lemacalculo} we know that for each critical point $c_{i}$ there exist a neighborhood 
$U_{i} \subseteq S^{1}$ of $c_i$ and a positive
constant $K_i$ such that for all $x \in U_{i} \setminus \{c_{i} \}$ we have
\begin{equation}\label{negS0}
 Sf(x)< -\dfrac{K_{i}}{(x-c_{i})^{2}} <0\,.
\end{equation}

Let us call $\mathcal{U}=\bigcup_{i=0}^{i=N-1} U_i$, and let $\mathcal{V}\subset S^1$ be an open set that contains none of the critical 
points of $f$ and such that $\mathcal{U}\cup \mathcal{V}=S^1$. Since $f$ is $C^3$, $M= \sup_{y \in \mathcal{V}}\big|Sf(y)\big|$ is finite. 
Let $\delta_n=\max_{0\leq j< q_{n+1}} |I_n^j|$. We know that $\delta_n\to 0$ as
$n\to \infty$, because $f$ is topologically conjugate to a rotation. We choose $n_1=n_1(f)$ so large that $\delta_n$ is smaller than the 
Lebesgue number of the covering $\{\mathcal{U}, \mathcal{V}\}$ of the circle for all $n\geq n_1$. Using the \emph{chain rule} for the Schwarzian 
derivative, we have for all $n\geq n_1$ and all $x\in I_{n}$ regular point of $f^{j}$
 \begin{equation}
 Sf^{j}(x)\;=\; \sum_{k=0}^{j-1} Sf(f^k(x))\left[Df^k(x)\right]^2 \ .
 \end{equation}
We can decompose this sum as $\Sigma_1^{(n)}(x) + \Sigma_2^{(n)}(x)$ where
 \begin{equation}\label{S1}
 \Sigma_1^{(n)}(x)\;=\; \sum_{k : I_n^k \subset \mathcal{U}} Sf(f^k (x))\left[Df^k (x)\right]^2 \ ,
\end{equation}
and $\Sigma_2^{(n)}(x)$ is the sum over the remaining terms.\\

Now we proceed through the following steps:
\begin{enumerate}
 \item[(i)] Since $I_n \subset \mathcal{U}$, the sum in the right-hand side of \eqref{S1} includes the term 
 with $k=0$, namely $Sf(x)$. Since all the other terms in \eqref{S1} are negative as well, and since $|x-c_0| \leq |I_n|$, we deduce from
 \eqref{negS0} that:
 \begin{equation}\label{S1final}
 \Sigma_1^{(n)}(x)\;<\; -\frac{K_1}{|I_n|^2} \,.
 \end{equation}
  \item[(ii)] Observe that,
  \begin{equation}
 \left|\Sigma_2^{(n)}(x)\right|\;\leq\; \sum_{I_n^k\subset \mathcal{V}} |Sf(f^k(x))|\left[Df^k(x)\right]^2.
\end{equation}
Assuming $n_1>n_0$, where $n_0=n_0(f)\in\nt$ is given by Lemma \ref{keylemma}, we know that there exists $K=K(f)>1$ such that
\begin{equation}\label{S2final}
\begin{aligned}
 \left|\Sigma_2^{(n)}(x)\right| &\leq \sum_{I_n^k\subset \mathcal{V}} |Sf(f^k(x))| K^{2} \dfrac{|I_{n}^k|^2}{|I_n|^2} \\
 &\leq M \dfrac{K^2}{|I_{n}|^2} \sum_{I_n^k\subset \mathcal{V}} |I_{n}^k|^2 \\
 &\leq  M \dfrac{K^2}{|I_{n}|^2} \max_{0 \leq k \leq j-1}|I_n^k| \sum_{I_n^k\subset \mathcal{V}} |I_{n}^k| \\
 &\leq  M \dfrac{K^2}{|I_{n}|^2}  \delta_{n}.
 \end{aligned}
\end{equation}
\end{enumerate}
 
Choosing $n_1$ so large that $K^2 M \delta_n < K_1$ for all $n\geq n_1$, we deduce from 
\eqref{S1final} and \eqref{S2final} that, indeed, $Sf^{j}(x)<0$ for all $j \in \{1, \cdots, q_{n+1} \}$ and for $n\geq n_1$.
\end{proof}

\section{Proof of main results}\label{sec:beaubounds}

In this final section we prove Theorem \ref{teobeau}, Theorem \ref{CRIuniversal}, Theorem \ref{scalingratios} and Coro\-llary~\ref{coro}.

For each critical point $c_i$ we consider its neighborhood $U_i$ given by Lemma \ref{lemacalculo}. Moreover, let $n_{1} \in \nt$ be given by Lemma \ref{negschwarz}. The following \emph{decomposition} will be crucial in the proof of Theorem \ref{CRIuniversal} given below (recall that, for a given $J\in\mathcal{P}_n$, we denote by $J^*$ the union of $J$ with its left and right neighbours in $\mathcal{P}_n$).

\begin{lemma}\label{lemadecomp} Given $\varepsilon>0$ there exists $n_2\in\nt$, $n_2=n_2(\varepsilon,f)>n_1$, with 
the following property: given $n \geq n_2$, $\Delta\in\mathcal{P}_n$ and $k\in\nt$ such that $f^j(\Delta)$ is contained in an element of
$\mathcal{P}_n$ for all $1\leq j \leq k$, we can write$$f^k|\Delta^*=\phi_{k}\circ\phi_{k-1}\circ...\circ\phi_1\,,$$where:
\begin{enumerate}
 \item For at most $3N+1$ values of $i\in\{1,...,k\}$, $\phi_i$ is a diffeomorphism with distortion bounded by $1+\varepsilon$.
 \item\label{casocrit} For at most $3N$ values of $i\in\{1,...,k\}$, $\phi_i$ is the restriction of $f$ to some interval contained in $U_i$. 
 \item For the remainder values of $i$, $\phi_i$ is either the identity or a diffeomorphism with negative Schwarzian derivative.
\end{enumerate}
\end{lemma}

In the proof we adapt the argument given in \cite[pages 352-353]{dFdM}.

\begin{proof}[Proof of Lemma \ref{lemadecomp}] Let $C_0=C_0(f) \geq 1$ be given by the Koebe distortion principle (Lemma \ref{koebe}). 
Let $C>1$ and $\mu \in (0,1)$ given by Theorem \ref{realbounds}. Let $\delta \in (0,1)$ be such that
$(1+\delta)^2 \exp(C_0\,\delta)<1+\varepsilon$, and let $n_{2} \in \nt$ be such that
 \[
 n_2 > n_{1}+\dfrac{4 \log(\delta \mu^{3/2}/C)}{\log \mu}\,.
 \]
Note that  $0<(\mu^{1/4})^{n_2-n_1} < \delta \mu^{3/2}/C$. Given $n\geq n_2$ consider$$m=m(n)=\left\lfloor \frac{n+n_1}{2} \right\rfloor,
$$the \emph{integer part} of $\frac{1}{2}(n+n_1)$. 
Let $\Delta$ and $k$ as in the statement, and consider $J_{m} \in \mathcal{P}_{m}$ such that $\Delta \subseteq J_{m}$, and consider also 
$J_{n_1} \in \mathcal{P}_{n_1}$ with $J_m \subseteq J_{n_1}$. Taking $n$ sufficiently large, we may assume that $\Delta^*\subset J_m$.

Let $s\geq 0$ be the smallest natural number such that $f^{s}(J_{n_1})$ contains a critical point of $f$.

\begin{claim}\label{claimdistfs} The distortion of $f^s$ on $\Delta^*$ is bounded by $1+\varepsilon$. 
\end{claim}

\begin{proof}[Proof of Claim \ref{claimdistfs}] The proof uses the Koebe Distortion Principle (Lemma \ref{koebe}). 
Replacing $n_1$ by $n_1+1$ if necessary, we may assume that $f^j(J_{n_1})\in\mathcal{P}_{n_1}$ for all $j\in\{0,...,s-1\}$. 
By the real bounds, the space $\tau$ of $\Delta^{*}$ inside $J_{m}^{*}$ is bounded from below by
 \[ 
 \tau \geq \dfrac{1}{C} \dfrac{|J_{m}|}{|\Delta^{*}|} \geq \dfrac{1}{C} 
 \left( \dfrac{1}{\mu} \right)^{\lfloor(n-m)/2\rfloor}>\frac{\mu}{C}\left( \dfrac{1}{\mu} \right)^{(n-m)/2}\ .
 \]
Since $m \leq \frac{n+n_1}{2}$, we have $n-m \geq n- \frac{n+n_1}{2} = \frac{n-n_1}{2}$, and then
 \begin{equation}\label{espacio}
 \dfrac{1}{\tau} \leq \frac{C}{\mu}\,\mu^{(n-m)/2} \leq \frac{C}{\mu}(\mu^{1/4})^{n-n_1} <\sqrt{\mu}\,\delta<\delta\,.
 \end{equation}
Now we estimate the sum $\ell$ of the lengths of the iterates of $J_{m}^*$ between $1$ and $s-1$. Since $\frac{n+n_1}{2} < m+1$, we have 
$m-n_1 > \frac{n-n_1}{2} -1$, and then for all $j \in \{ 0, \cdots, s-1\}$:
$$\big|f^{j}(J_m^*)\big|\leq \mu^{\lfloor(m-n_1)/2\rfloor}\big|f^{j}(J_{n_1}^*)\big| \leq 
   (\mu^{1/4})^{n-n_1} \left( \dfrac{1}{\mu}\right)^{3/2}\big|f^{j}(J_{n_1}^*)\big|\leq \dfrac{\delta}{C} \big|f^{j}(J_{n_1}^*)\big|\,.$$
Therefore:
\begin{equation}\label{iterados}
   \ell=\sum_{j=0}^{s-1} |f^{j}(J_{m}^*)| <\dfrac{3\delta}{C}<\delta\,,
  \end{equation}
since $\sum_{j=0}^{s-1} |f^{j}(J_{n_1}^*)| <3$ by combinatorics (and assuming $C>3$). From inequalities (\ref{espacio}), (\ref{iterados}) and Koebe distortion principle (see \eqref{constkoebe}) we get that the distortion 
on $\Delta^{*}$ is bounded from above
by$$\left(1+ \delta \right)^2 \exp(C_0\,\delta)<1+\varepsilon\,.$$\end{proof}

To prove Lemma \ref{lemadecomp} we decompose the orbit of $\Delta^*$ under $f$ according to the following algorithm. 
For each $i\in\{0,1,...,k-1\}$ we have two cases to consider:
\begin{enumerate}
 \item If $f^i(J_{n_1})$ does not contain any critical point of $f$, we define the corresponding $\phi$ to be $f^{s}$, where $s\geq 1$ is the smallest
natural such that $f^{i+s}(J_{n_1})$ contains a critical point of $f$. Arguing as in Claim \ref{claimdistfs} above, we see that this case belongs
to the first type of components in the statement.
\item If $f^i(J_{n_1})$ contains a critical point $c$ of $f$ we may assume, by taking $n_2$ large enough, that 
$f^i(\Delta^*)\subset I_{n_1}(c)\cup I_{n_1+1}(c)$.
We have two sub-cases to consider:

\begin{enumerate}
\item [(i)] If $f^i(\Delta^*)$ does not contain $c$ (and therefore no other critical point) let $s\geq 1$ be the smallest natural 
such that $f^{i+s}(\Delta^*)$
contains a critical point of $f$, and we define the corresponding $\phi$ to be $f^{s}$. By Lemma \ref{negschwarz} 
(and the fact that composition of diffeomorphisms with negative Schwarzian derivative is a diffeomorphism with 
negative Schwarzian derivative too) this case belongs to the third type of components in the
statement.
\item [(ii)] If the critical point belongs to $f^i(\Delta^*)$ we define the corresponding $\phi$ to be just a single
iterate of $f$ (and this sub-case belongs to the second type of components in the statement).
\end{enumerate}
\end{enumerate}

Note finally that, by combinatorics, the first case happens at most $3N+1$ times, while the second case occurs at most $3N$ times.
\end{proof}

With Lemma \ref{lemadecomp} at hand, we are ready to prove our main results.

\begin{proof}[Proof of Theorem \ref{CRIuniversal}] Theorem \ref{CRIuniversal} follows at once from the decomposition obtained in 
Lemma \ref{lemadecomp}, by combining Remark
\ref{remdistandcrd}, Lemma \ref{contracts} and Item \eqref{itemcross} of Lemma \ref{lemacalculo}. The constant $B$ depends only 
on the number and order of the critical points of $f$, but not on $f$ itself. It is in fact enough to consider $B=(1+1/2)^{2(3N+1)}(9d^2)^{3N}$.
\end{proof}

\begin{proof}[Proof of Theorem \ref{teobeau}] As explained in Section \ref{statements}, the proof of Theorem \ref{teobeau} is the 
same as the proof of the real bounds (Theorem \ref{realbounds}) given by the first two authors in \cite[Section 3]{EdF}, but replacing the Cross-Ratio Inequality with Theorem \ref{CRIuniversal}.
\end{proof}

\begin{proof}[Proof of Theorem \ref{scalingratios}] This is clearly a special case of Theorem \ref{teobeau}.
\end{proof}

\begin{proof}[Proof of Corollary \ref{coro}] Here we merely sketch the proof (the details are tedious repetitions of arguments in \cite{EdF}).
The proof uses the notion of \emph{fine grids} given in \cite[Definition 5.1]{EdF} and the criterion for quasi-symmetry given in
\cite[Proposition 5.1]{EdF}. Let $\big\{\mathcal{Q}_n(f)\big\}_{n \geq 0}$ be the fine grid constructed in \cite[Proposition 5.2]{EdF}, 
and let $B>1$ and $n_0=\max\big\{n_0(f),n_0(g)\big\}$ be given by Theorem \ref{teobeau}. Then for all $n \geq n_0$, adjacent atoms of 
$\mathcal{P}_n^f$ are comparable by the constant $B$, and the same is valid for adjacent atoms of $\mathcal{P}_n^g$. Consider the sequence 
$\big\{\mathcal{Q}'_n(f)\big\}_{n \geq n_0}$ of partitions of $J_{n_0}^{f}=I_{n_0}^{f} \cup I_{n_0+1}^{f}$ given by 
$\mathcal{Q}'_n(f)=\{\Delta\in\mathcal{Q}_n(f):\Delta \subset J_{n_0}^{f}\}$. Then $\big\{\mathcal{Q}'_n(f)\big\}_{n \geq n_0}$ is a fine 
grid restricted to $J_{n_0}^{f}$, and its fine grid constants depend only on $B$, $N$ and $d$, and therefore are universal. By 
\cite[Proposition 5.1]{EdF}, it follows that $h|_{J_{n_0}^{f}}$ has quasi-symmetric distortion bounded by $K_0=K_0(B,N,d)$ (a universal 
constant). In particular, we have $\sigma_{h}(x) \leq K_0$ for all $x \in J_{n_0}^{f}$. It now follows from Theorem \ref{CRIuniversal} 
that $\sigma_{h}(x) \leq K_1$ for all $x \in S^1$, for some universal constant $K_1=K_1(N,d)$.
\end{proof}

\appendix

\section{Proofs of auxiliary results}\label{appA}

In this appendix, we prove the three auxiliary lemmas stated without proof in the main text: Lemma \ref{contracts}, Lemma \ref{lemacalculo} 
and Lemma \ref{lemapartition}. All of them are well known, but we provide proofs for the sake of completeness of exposition, and as a courtesy 
to the reader. Let us start with the following observation.

\begin{lemma}\label{nucleoschwarz} The kernel of the Schwarzian derivative is the group of M\"obius transformations. Moreover, if $\phi$ 
is a M\"obius transformation and $f$ is any $C^3$ map, then $S(\phi \circ f)=Sf$.
\end{lemma}

\begin{proof}[Proof of Lemma \ref{nucleoschwarz}] On one hand, the fact that the Schwarzian derivative vanish at M\"obius transformations 
is a straightforward computation. On the other hand, given an increasing $C^3$ map $\phi$ without critical points on some interval $I$, 
consider the $C^2$ map $g$ defined by $g=(D\phi)^{-1/2}$. A straightforward computation gives the identity:$$S\phi=-2\left(\frac{D^2g}{g}\right).$$

In particular $S\phi\equiv0$ if and only if $D^2g\equiv0$, and then there exist real numbers $a$ and $b$ such that $g(x)=ax+b$, that is, 
$D\phi(x)=1/(ax+b)^2$. By integration we get:$$\phi(x)=\left(\frac{-1}{a}\right)\left(\frac{1}{ax+b}\right)+c\,,$$for some real number $c$. 
In particular, $\phi$ is a M\"obius transformation.

Finally, by the \emph{chain rule} for the Schwarzian derivative of the composition of two functions:
$$S(\phi \circ f)(x)=Sf(x)+S\phi(f(x))(Df(x))^{2}\,,$$we see at once that if $\phi$ is a M\"obius transformation, 
we have $S\phi \equiv 0$ and then $S(\phi \circ f)=Sf$.
\end{proof}

Let us point out that the change of variables used in the proof of Lemma \ref{nucleoschwarz} was already used by Yoccoz in \cite{yoccoz}. 
With Lemma \ref{nucleoschwarz} at hand, we are ready to prove Lemma \ref{contracts}, stated in Section \ref{prelim}.

\begin{proof}[Proof of Lemma \ref{contracts}] The proof is the one given in \cite[Section IV.1]{dMvS}. Let $M =[b,c ]\subseteq T=[a,d]$. 
Let us call $L$ and $R$ the two connected components of $T \setminus M$. Let $\phi$ be the (unique) M\"obius transformation such that 
$\phi(f(a))=a,$ $\phi(f(b))=b$ and $\phi(f(d))=d$. Note that $\phi \circ f$ is a $C^3$ diffeomorphism with negative Schwarzian derivative, 
since $S(\phi \circ f)=Sf<0$ by Lemma \ref{nucleoschwarz}.

We claim that $\phi (f (c))>c$. Indeed, if this is not true, then by the Mean Value Theorem there exist $z_0 \in [a,b]$, $z_1 \in [b,c]$ 
and $z_2 \in [c,d]$ such that
\[ D(\phi \circ f)(z_0) = \dfrac{\phi(f(a))-\phi(f(b))}{a-b} =1 , \ \ D(\phi \circ f)(z_1) = \dfrac{\phi(f(c))-\phi(f(b))}{c-b} \leq 1 \] and
\[ D(\phi \circ f)(z_2) = \dfrac{\phi(f(d))-\phi(f(c))}{d-c} \geq 1. \]

If\footnote{In the particular case $z_1=z_0$, we obtain $z_1=z_0=b$, and then $D\big(\phi \circ f\big)(b)=1$ and $\phi (f (c))=c$. This implies that $D\big(\phi \circ f\big)(c)<1$ (otherwise, the Minimum Principle would imply that $D\big(\phi \circ f\big)(x)>1$ for all $x \in (b,c)$, which is impossible since $\phi \circ f$ fixes both $b$ and $c$). Again, this contradicts the Minimum Principle since $c\in(b,z_2)$. The remaining case $z_1=z_2$ is analogous.}$z_1\in(z_0,z_2)$, the previous inequalities contradict the Minimum Principle for diffeomorphisms with negative Schwarzian derivative \cite[Section II.6, Lemma 6.1]{dMvS}. Therefore, $\phi(f(c))>c$ as claimed. With this at hand we get:
\[\crd(\phi\circ f; M,T)= \dfrac{\big[\phi\big(f(M)\big),\phi\big(f(T)\big)\big]}{[M,T]} = \dfrac{\big|M \cup L\big|\,\big|\phi\big(f(c)\big)-d\big|}{\big|R\big|\,\big|a-\phi\big(f(c)\big)\big|} <1\,.\]

Since $\phi$ is a M\"obius transformation, $\crd(\phi \circ f; M,T)=  \crd(f;M,T)$ and the lemma is proved.
\end{proof}

\begin{proof}[Proof of Lemma \ref{lemacalculo}] From Definition \ref{naoflat} there exists a neighborhood of the critical point $c$ 
such that $f(x)= g\big(\phi(x)\big) +f(c)$, where $g$ is the map $x \mapsto x^{d}$ and $\phi$ is a $C^{3}$ diffeomorphism with $\phi(c)=0$. 
The chain rule for the Schwarzian derivative gives $Sf = Sg(\phi)(D \phi)^2+S\phi$.

Since $Sg(x)= -\dfrac{(d^2-1)}{2 x^2}$, we get: 
\[
Sg(\phi(x))(D \phi(x))^2= -\dfrac{1}{2}(d-1)(d+1)\left( \dfrac{D \phi(x)}{\phi(x)} \right)^{2} \leq - \dfrac{A}{(\phi(x))^{2}},
\]
where $A= \dfrac{1}{2}(d-1)(d+1)\min_{x}\big|D\phi(x)\big|>0$. In particular:$$Sf(x)<\frac{-A+S\phi(x)\big(\phi(x)\big)^2}{\big(\phi(x)\big)^2}\,.$$

On the other hand, since $\phi$ is a diffeomorphism, $|S\phi(x)|<M$ for some $M>0$. Then we can choose $\delta>0$ such that for all $x \in (c-\delta, c+\delta)$ 
we have $|\phi(x)|< \sqrt{\frac{A}{M}}$, and this implies that $Sf<0$ in  $(c-\delta,c+\delta)\setminus\{c\}$. Finally, since $\phi$ is bi-Lipschitz we have $|\phi(x)| \asymp|x-c|$ and we obtain Item \eqref{itemSf}.

Item \eqref{itemDf} follows at once from Taylor Theorem since:
\[
 \lim_{x \rightarrow c} \left( \dfrac{D f(x)}{|x-c|^{d-1}} \right) = d(D \phi (c))^{d} >0\,.
\]

With Item \eqref{itemDf} at hand we prove Item \eqref{itemdist}. Let $J=(a,b)\subseteq U$. By symmetry it is enough to consider the following two cases:
\begin{enumerate}
\item[(i)] $c \leq a <b$: In this case we have for any $x\in(a,b)$ that
\begin{equation*}
\begin{aligned}
\dfrac{Df(x) |J|}{|f(J)|} &\leq \dfrac{\beta(x-c)^{d-1}(b-a)}{\alpha\int_{a}^{b} (t-c)^{d-1}dt} \\
&\leq \left(\dfrac{\beta d}{\alpha}\right) \dfrac{(b-c)^{d-1} (b-c-a+c)}{(b-c)^{d} - (a-c)^{d}} \\
&= \left(\dfrac{\beta d}{\alpha}\right) \left( 1 + \dfrac{(a-c)^{d}-(b-c)^{d-1}(a-c)}{(b-c)^d-(a-c)^d} \right)\\
&\leq \dfrac{\beta d}{\alpha}<3d/2.
\end{aligned}
\end{equation*}
\item[(ii)] $a< c< b$: Without loss of generality, we may assume that $|a-c| < |c-b|$. If $x \in J$, then$$\dfrac{Df(x) |J|}{|f(J)|} 
\leq \dfrac{\beta |x-c|^{d-1}|b-a|}{\int_{c}^{b} Df(t)\,dt} 
\leq \dfrac{2 \beta |b-c|^{d}}{\int_{c}^{b} \alpha(t-c)^{d-1}dt}
= \dfrac{2 \beta d}{ \alpha}<3d.$$
\end{enumerate}

Finally, to prove Item \eqref{itemcross}, let us call $L,R$ the two connected components of 
$T \setminus M$. By the Mean Value Theorem there exist $z_{0} \in L$ and $z_{1} \in R$ such that
$$\crd(f;M,T) = \dfrac{Df(z_{0})\,Df(z_{1})\,|L \cup M|\,|M \cup R|}{\big|f(L \cup M)\big|\,\big|f(M \cup R)\big|}\,.$$

Since $z_0\in L\cup M$ and $z_1\in R\cup M$ we obtain from Item \eqref{itemdist} that$$\crd(f;M,T)\leq (3d)^2.$$
\end{proof}

\begin{proof}[Proof of Lemma \ref{lemapartition}] Since the families $\mathcal{P}_n$ are dynamically defined, and since any multicritical circle 
map with irrational rotation number is topologically conjugate to a rigid rotation (see Yoccoz's Theorem \ref{yoccoztheorem} in Section \ref{prelim}) we will assume in this proof that $f$ is itself the rigid rotation in the unit circle of angle $2\pi\rho$, where $\rho\in [0,1)$ is an irrational number. Moreover, in order to simplify the notation, we normalize the unit circle to have total length equal to $1$ (and then $f$ is just the rotation of angle $\rho$). Being irrational, $\rho$ has an infinite continued-fraction expansion, say $\rho = [a_0,a_1, \cdots]$.

We claim that for all $n\in\nt$, if $\{p_{n}/q_{n} \}$ is the sequence obtained by truncating the continued-fraction expansion at level $n-1$, 
we have:
\begin{equation}\label{soma1}
q_{n}p_{n+1}-q_{n+1}p_{n}= (-1)^{n}\,.
\end{equation}
Indeed, note that $q_{0}p_{1}-q_{1}p_{0}= 1$ and that $q_{1}p_{2}-q_{2}p_{1}= a_{0}a_{1} -a_{1}a_{0}-1=-1$. Let us suppose now that 
$q_{n}p_{n+1}-q_{n+1}p_{n}= (-1)^{n}$. Then:
\begin{equation*}
   \begin{aligned}
     q_{n+1}p_{n+2}-q_{n+2}p_{n+1} &= q_{n+1}(a_{n+1}p_{n+1}+p_{n})- (a_{n+1}q_{n+1}+q_{n})p_{n+1} \\
      &= a_{n+1}q_{n+1}p_{n+1}+p_{n}q_{n+1}-a_{n+1}q_{n+1}p_{n+1} -q_{n}p_{n+1} \\
      &=-(q_{n}p_{n+1}-q_{n+1}p_{n})=(-1)^{n+1}\,,\quad\mbox{as claimed.}
   \end{aligned}
  \end{equation*}

The arithmetical properties of the continued fraction expansion described in \S \ref{prelim} imply that, for any point $x \in S^1$, the iterates 
$\{f^{q_n}(x)\}_{n \in \nt}$ are the \emph{closest returns} of the orbit of $x$ under the rigid rotation $f$, in the following 
sense:$$d\big(x,f^{q_n}(x)\big)<d\big(x,f^j(x)\big)\quad\mbox{for any}\quad j\in\{1,...,q_n-1\}$$where $d$ denote the standard distance in $S^1$. 
In particular, all members of the family$$\big\{I_n, f(I_n), \cdots, f^{q_{n+1}-1}(I_n)\big\}$$are  pairwise disjoint, and all members in the 
family$$\big\{I_{n+1}, f(I_{n+1}), \cdots, f^{q_{n}-1}(I_{n+1})\big\}$$are pairwise disjoint too. Moreover, we claim that any two members in the 
union of these families (and recall that this union is precisely the definition of $\mathcal{P}_n$) are disjoint. Indeed, suppose, by 
contradiction, that there exist $i<q_{n+1}$ and $j<q_{n}$ such that $f^{i}(I_{n}) \cap f^{j}(I_{n+1}) \neq \O$. Without loss of generality,
we may assume that $i<j=i+l$, for some $l<q_{n}$, and that the $q_n$-th iterate of every point $x \in S^{1}$ is on the right-hand side of $x$, 
and consequently the $q_{n+1}$-th iterate is on the left-hand side of $x$. We have three possible cases to consider:
\begin{itemize}
\item If $f^{i}(I_{n}) \subseteq f^{j}(I_{n+1})$, then $f^{j}(I_{n+1})$ intersects $f^{i}(I_{n+1})$ and this is impossible as explained above.
\item If $f^{j}(I_{n+1}) \subseteq f^{i}(I_{n})$, then the point $f^{j}(c)=f^{i+l}(c)$ is closer to $f^{i}(c)$ than $f^{i+q_n}(c)$, which is impossible since $l<q_{n}$.
\item If both differences between $f^{j}(I_{n+1})$ and $f^{i}(I_{n})$ are non-empty and connected, then we have two sub-cases: either $f^{j}(c) \in f^{i}(I_{n})$ or $f^{j+q_{n+1}}(c) \in f^{i}(I_{n})$. In the first case, the point $f^{j}(c)=f^{i+l}(c)$ is closer to $f^{i}(c)$ than $f^{i+q_{n}}(c)$, and since $l<q_{n}$ this is a contradiction. In the second case, the point $f^{i+q_{n}}(c)= f^{j}(f^{q_{n}+i-j}(c))$ is closer to 
  $f^{j}(c)$ than $f^{j+q_{n+1}}(c)$, which again is impossible since $q_{n}+i-j<q_{n+1}$.
  \end{itemize}
Therefore, any two members of $\mathcal{P}_n$ are disjoint, as claimed.

Finally, since we are assuming that $f$ is the rigid rotation of angle $\rho$ in the (normalized) unit circle, the lengths of the intervals 
$I_n$ and $I_{n+1}$ are $|q_{n}\rho -p_{n}| = q_{n}|\rho - p_{n}/q_{n}|$ and $q_{n+1}|p_{n+1}/q_{n+1} -\rho|$ respectively. Therefore, the 
total length of the union of the members of $\mathcal{P}_n$ is equal to:
  \begin{equation*}
   \left|q_{n}q_{n+1}\left( \dfrac{p_{n+1}}{q_{n+1}} - \dfrac{p_{n}}{q_{n}} \right)\right|= |q_{n}p_{n+1}-p_{n}q_{n+1}|.
   \end{equation*}
By \eqref{soma1}, this absolute value is equal to $1$, that is, the union of the members of $\mathcal{P}_n$ is a compact set of full Lebesgue 
measure, and therefore it covers the whole circle.
\end{proof}

\section*{Acknowledgements}

We wish to thank the warm hospitality of IMPA, where part of this paper was written.

\end{document}